\documentclass[letterpaper,11pt]{amsart}

\usepackage[margin=1.2in]{geometry}
\usepackage{mathrsfs,amsmath,amsthm,amssymb}
\usepackage{xspace,xcolor}
\usepackage[breaklinks,colorlinks,citecolor=teal,linkcolor=teal,urlcolor=teal,pagebackref,hyperindex]{hyperref}
\usepackage[alphabetic]{amsrefs}
\usepackage[all,cmtip]{xy}
\usepackage{mathtools}
\hypersetup{final}
\numberwithin{equation}{section}

\newtheorem*{theorem*}{Theorem}
\newtheorem{theorem}{Theorem}[section]
\newtheorem{lemma}[theorem]{Lemma}
\newtheorem{proposition}[theorem]{Proposition}
\newtheorem{corollary}[theorem]{Corollary}

\newtheorem{thmx}{Theorem}
\newtheorem{innercustomthm}{Theorem}
\newenvironment{customthm}[1]
  {\renewcommand\theinnercustomthm{#1}\innercustomthm}
  {\endinnercustomthm}

\theoremstyle{definition}
\newtheorem{definition}[theorem]{Definition}

\theoremstyle{remark}
\newtheorem*{remark}{Remark}

\newtheorem{claim}{Claim}
\newtheorem*{ack}{Acknowledgments}
\newtheorem*{notation}{Notation}

\DeclareMathOperator{\Eff}{Eff}

\DeclareMathOperator{\Supp}{Supp}
\DeclareMathOperator{\Hom}{Hom}

\DeclareMathOperator{\N}{N}
\DeclareMathOperator{\NE}{NE}
\DeclareMathOperator{\Bl}{Bl}

\begin{document}
\title{Numerical dimension and locally ample curves}
\author{Chung-Ching Lau}
\date{}
\address{Department of Mathematics, University of Utah, Salt Lake City, UT 84112, USA}
\thanks{Research partially suppported by NSF FRG grant DMS-1265285}
\email{{\tt lau@math.utah.edu}}

\maketitle

\begin{abstract} 
In the paper \cite{Lau16},
it was shown that the restriction of a pseudoeffective divisor $D$ to a so-called nef subvariety $Y$ (e.g. $Y$ is lci in $X$ and has nef normal bundle) is pseudoeffective.
Assuming the normal bundle is ample and that $D|_Y$ is not big,
we prove that the numerical dimension of $D$ is bounded above by that of its restriction, i.e. $\kappa_{\sigma}(D)\leq \kappa_{\sigma}(D|_Y)$.
The main motivation is to study the cycle classes of ``positive'' curves:
we show that
the cycle class of a curve with ample normal bundle lies in the interior of the cone of curves,
and the cycle class of an ample curve lies in the interior of the cone of movable curves. 
We do not impose any condition on the singularities on the curve or the ambient variety.
For locally complete intersection curves in a smooth projective variety,
this is the main result of Ottem \cite{Ott16}.
The main tool in this paper is the theory of $q$-ample divisors.
\end{abstract}

\section{Introduction}
This paper deals with subvarieties (of projective variety) which manifest positivity property.
Recall that a divisor $D$ is $q$-\textit{ample} if
for any $\mathscr{F}$ there is an $m_0$ such that
\[
H^{i}(X,\mathscr{F}\otimes\mathscr{O}_X(mD))=0
\,\,\,\, \forall m\geq m_0.
\]

Let $X$ be a projective variety, let $Y$ be a subvariety of $X$ of codimension $r$ and let $\tilde{X}\rightarrow X$ be the blowup morphism of $X$ along $Y$, with exceptional divisor $E$.
We call $Y$ a \textit{locally ample} subvariety of $X$ if $\mathscr{O}_{E}(E)$ is $(r-1)$-ample.
If $Y$ is lci in $X$, being locally ample is equivalent to having ample normal bundle.
We call $Y$ an \textit{ample} subvariety of $X$ if $\mathscr{O}_{\tilde{X}}(E)$ is $(r-1)$-ample (The notion of an ample subvariety was introduced in \cite{Ott12}).
We call $Y$ a \textit{nef} subvariety of $X$ if $\mathscr{O}_{E}(mE+A)$ is $(r-1)$-ample for $m\gg 0$, where $A$ is an ample divisor.
If $Y$ is l.c.i. in $X$, being nef is the same as having nef normal bundle.

In \cite{Lau16}, we showed that the restriction of a pseudoeffective divisor to a nef subvariety is pseudoeffective.
In this paper, we shall study how the numerical dimension of the classes on the boundary of $\overline{\Eff}^{1}(X)$ behave under the restriction $\iota^{*}: \overline{\Eff}^{1}(X)\rightarrow \overline{\Eff}^{1}(Y)$, assuming $Y$ is locally ample.

Nakayama showed that if $H$ is a smooth ample divisor of a smooth projective variety $X$ and $\eta\in\N^{1}(X)_{\mathbf{R}}$ is not big,
then $\kappa_{\sigma}(\eta)\leq \kappa_{\sigma}(\eta|_{H})$ \cite[Proposition 2.7(5)]{Nak04}.
On the other hand,
Ottem showed that if $X$ is a smooth projective variety,
$Y$ is a l.c.i. subvariety with ample normal bundle and $\eta\in\N^{1}(X)_{\mathbf{R}}$ satisfies $\eta|_{Y}=0$,
then $\kappa_{\sigma}(\eta)=0$ \cite[Theorem 1]{Ott16}.
This was a conjecture due to Peternell \cite[Conjecture 4.12]{Pet12}.
The following theorem generalizes both of the above results.
\begin{thmx}\label{thm: 3}
Let $\iota: Y\hookrightarrow X$ be a locally ample subvariety of codimension $r$ of a projective variety $X$.
If $\eta\in\N^{1}(X)_{\mathbf{R}}$ is a pseudoeffective class such that $\eta|_Y$ is not big, then
$\kappa_\sigma(\eta)\leq \kappa_\sigma(\eta|_Y)$.
\end{thmx}

From this, we deduce the following result (see Theorem \ref{thm: C}). 
\begin{thmx}
Let $Y$ be a locally ample subvariety of $X$ and let $f:X\rightarrow Z$ be a surjective morphism from $X$ to a projective variety $Z$.
If $\dim f(Y)<\dim Y$, then $f|_Y: Y\rightarrow Z$ is also surjective, i.e. $f(Y)=Z$.
\end{thmx}
One can regard these results as hints that it is natural to study the notion of locally ample subvariety.

We now turn our focus to the main application of Theorem \ref{thm: 3}.

It seems interesting to ask how the positivity of the normal bundle of a subvariety influences the positivity of the underlying cycle class of the subvariety.
The divisor case is well-known.
For example, ample divisors generate an open cone in $\N^1(X)_{\mathbb{R}}$, called the ample cone. 
The closure of the ample cone is dual to the closure of the cone generated by curves in $X$ (Kleiman).
Furthermore, an effective Cartier divisor with ample normal bundle is big \cite[Theorem III.4.2]{Har70}.
In this paper, we want to see whether similar properties hold for curves.
Boucksom, Demailly, P\u{a}un and Peternell \cite{BDPP} showed that the closure of the cone of effective divisors in $\N^1(X)_{\mathbb{R}}$, called the pseudoeffective cone, is dual to the closure of the cone generated by strongly movable curves, called the movable cone of curves.
Using this result, one can show that the cycle class of a nef curve (in particular a curve with nef normal bundle) lies in the movable cone of curves (\cite[Theorem 4.1]{DPS}, \cite[Theorem 1.3]{Lau16}).
By analogy to the divisor case, it is natural to pose the following question: given a locally ample (resp. ample) curve, does the cycle class of the curve lies in the \textit{interior} of the cone of curves (resp. movable cone of curves)?
In this paper, we give a positive answer to this question.

\begin{thmx}\label{thm: thmB}
Let $X$ be a projective variety and let $Y$ be a locally ample curve in $X$. Then $[Y]\in\N_{1}(X)_{\mathbf{R}}$ is big, i.e. it lies in the interior of cone of curves.
Furthermore, if $Y$ meets all prime divisors of $X$, e.g. $Y$ is ample, then $[Y]$ lies in the interior of the movable cone of curves.
\end{thmx}
Following an observation of Peternell \cite[Conjecture 4.1]{Pet12},
Ottem already deduced that the cycle class of a locally complete intersection curve with ample normal bundle in a smooth projective variety lies in the interior of the cone of curves (\cite[Theorem 2]{Ott16}).
Indeed, if $\eta\in\N^{1}(X)_{\mathbf{R}}$ is nef and $\eta|_Y=0$, then the conjecture says $\kappa_{\sigma}(\eta)=0$,
which forces $\eta=0$.
Theorem \ref{thm: thmB} improves upon Ottem's result by removing any restrictions on smoothness of $X$ and $Y$.
Our proof is different from Ottem's in the sense that the theory of $q$-ample divisors is used here.

\begin{notation}
We work over a field of characteristic zero.
A variety is meant to be an integral scheme.
A curve is meant to be an integral scheme of dimension $1$.
\end{notation}
\begin{ack}
I would like to thank my advisor, Tommaso de Fernex, for his many comments that improves the exposition.
I would also like to thank Brian Lehmann and John Christian Ottem for helpful discussions,
and the referee for his careful reading of the paper and many useful suggestions.
This is part of the author's PhD thesis at University of Utah.
\end{ack}
\section{Preliminaries}
In this section, we shall recall the necessary definitions and tools needed.

\subsection{Dualizing sheaf}
\begin{definition}[Dualizing sheaf {\cite[p.241]{Har77}}]
Let $X$ be a projective scheme of dimension $n$.
A \textit{dualizing sheaf} for $X$
is a coherent sheaf $\omega_X$,
together with a trace map
$t:H^{n}(X,\omega)\rightarrow k$ to the ground field $k$, such that for any coherent sheaf $\mathscr{F}$ on $X$
the natural pairing
\[
H^{n}(X,\mathscr{F})\times \Hom(\mathscr{F},\omega_X)\rightarrow H^{n}(X,\omega_X),
\]
followed by $t$,
is perfect.
\end{definition}
\begin{proposition}\cite[Proposition 7.2, 7.5]{Har77}
Let $X$ be a projective scheme of dimension $n$.
Then the dualizing sheaf for $X$ exists and is unique up to unique isomorphism.
\end{proposition}
We now show that a dualizing sheaf can be embedded into a sufficiently ample line bundle.
The proof can be found in the proof of \cite[Theorem 9.1]{Tot13}, but we include here for the sake of convenience.
\begin{lemma}[Embedding a dualizing sheaf into a line bundle]\label{lemma:dualizing}
Let $X$ be a projective variety of dimension $n$.
Then $\omega_X$ is torsion-free.
Moreover, given an ample divisor on $H$, there is $l$ such that there is an embedding $\omega_X\hookrightarrow \mathscr{O}_{X}(lH)$.
\end{lemma}
\begin{proof}
Let us first show that $\omega_X$ is torsion-free.
Indeed, let $\mathscr{T}\subset\omega_X$ be the torsion subsheaf.
Then
\[
\Hom(\mathscr{T},\omega_X)\cong H^{n}(X,\mathscr{T})^{\vee}= 0.
\]
The last equality follows from the fact that $\mathscr{T}$ is supported at a proper closed subscheme of $X$.

As $\omega_X$ is generically a line bundle,
$\omega_X^{\vee}\neq 0$.
For $l$ large,
there is a nontrivial section $s\in H^{0}(X,\omega^{\vee}\otimes \mathscr{O}_X(lH))$.
This induces a nontrivial map $\omega_X \rightarrow \mathscr{O}_X(lH)$, which has to be an injection, since $\omega_X$ is torsion free of rank $1$.
\end{proof}

\subsection{\texorpdfstring{$q$}{q}-ample divisors}

The main tool used in this paper is the theory of $q$-ample divisors, developed by Sommese \cite{Som78}, Demailly-Peternell-Schneider \cite{DPS} and Totaro \cite{Tot13}.
Let us recall its definition.
\begin{definition}[$q$-ample line bundle {\cite{DPS},\cite{Tot13}}]
Let $X$ be a projective scheme. A line bundle bundle $\mathscr{L}$ is \textit{$q$-ample} 
if for any coherent sheaf $\mathscr{F}$ on $X$, there is an $m_{0}$ such that
\[
H^{i}(X,\mathscr{F}\otimes\mathscr{L}^{\otimes m})=0
\]
for $i>q$ and $m>m_0$.
\end{definition}

We shall give the definition of a Koszul-ample line bundle. 
The details are not very important in this paper,
but they are included for the sake of completeness.
Koszul-ample line bundle comes up in the definition of a $q$-T-ample line bundle, which we shall give shortly.
One useful fact is that any large tensor power of an ample line bundle is $2n$-Koszul-ample, where $n$ is the dimension of the underlying projective scheme \cite{Bac86}.
\begin{definition}[Koszul-ampleness  {\cite[Section 1]{Tot13}}]
Let $X$ be a projective scheme of dimension $n$,
and that the ring of regular function $\mathscr{O}(X)$ on $X$ is a field (e.g. $X$ is connected and reduced).
Given a very ample line bundle $\mathscr{O}_X(1)$, 
we say that it is \textit{$N$-Koszul ample} if the homogeneous coordinate ring $A=\bigoplus_{j} H^{0}(X,\mathscr{O}_X(j))$ is $N$-Koszul, i.e. there is a resolution
\[
\cdots\rightarrow M_1 \rightarrow M_0\rightarrow k\rightarrow 0
\]
where $M_i$ is a free $A$-module, generated in degree $i$, where $i\leq N$.
\end{definition}

\begin{definition}[$q$-T-ampleness {\cite[Definition 6.1]{Tot13}}]
Let $X$ be a projective variety of dimension $n$. 
We fix a $2n$-Koszul-ample line bundle $\mathscr{O}_{X}(1)$ on $X$.
We say that a line bundle $\mathscr{L}$ is \textit{$q$-T-ample} if there is a positive integer $N$, such that
\[
H^{q+i}(X,\mathscr{L}^{\otimes N}\otimes \mathscr{O}_X(-n-i))=0,
\]
for $1\leq i\leq n-q$.
\end{definition}
The following theorem is the key technical theorem in Totaro's paper.
\begin{theorem}\cite[Theorem 6.3]{Tot13}\label{thm:q-T-ample}
The notion of $q$-ampleness and $q$-T-ampleness are equivalent.
\end{theorem}

\begin{definition}[$q$-ample $\mathbf{R}$-Cartier $\mathbf{R}$-divisors]
Let $X$ be a projective scheme.
An $\mathbf{R}$-Cartier $\mathbf{R}$-divisor $D$ on $X$ is \textit{$q$-ample} if $D$ is numerically equivalent to $cL+A$ with $L$ a $q$-ample line bundle, $c\in \mathbf{R}_{>0}$, $A$ an ample $\mathbf{R}$-Cartier $\mathbf{R}$-divisor.
\end{definition}
Based on the work of Demailly, Peternell and Schneider, Totaro also proved that 
\begin{theorem}[{\cite[Theorem 8.3]{Tot13}}]\label{thm:open}
An integral divisor is $q$-ample if and only if its associated line bundle is $q$-ample.
The $q$-ample $\mathbf{R}$-Cartier $\mathbf{R}$-divisors in $\N^{1}(X)_{\mathbf{R}}$ defines an open cone (but not convex in general) and that the sum of a $q$-ample $\mathbf{R}$-Cartier $\mathbf{R}$-divisor and an $r$-ample $\mathbf{R}$-Cartier $\mathbf{R}$-divisor is $(q+r)$-ample.
\end{theorem}
\begin{remark}
Totaro's paper relies on \cite[Theorem 1.4]{DPS} for a proof of the fact that q-ampleness descends
to numerical equivalence classes, but the proof given in \cite{DPS} only works in
the smooth case. 
For projective varieties in general the claim is treated in Greb and K\"uronya's paper \cite[Theorem 2.17]{GK15}.
\end{remark}
\begin{theorem}[{\cite[Theorem 9.1]{Tot13}}]\label{theorem: n-1 ample}
Let $X$ be a projective variety of dimension $n$. 
A line bundle $\mathscr{L}$ on $X$ is $(n-1)$-ample if and only if
$[\mathscr{L}^{\vee}]\in N^{1}(X)$ does not lie in the pseudoeffective cone.
\end{theorem}
\begin{definition}[$q$-almost ample]
Let $X$ be a projective scheme and let $A$ be an ample divisor on $X$.
We say that a $\mathbf{R}$-Cartier $\mathbf{R}$-divisor $D$ is $q$-\textit{almost ample} if $D+\epsilon A$ is $q$-ample for all $0<\epsilon \ll 1$.
\end{definition}

\subsection{\texorpdfstring{$\sigma$}{sigma}-dimension}

Let us start with the definition of the $\sigma$-dimension of an $\mathbf{R}$-Cartier $\mathbf{R}$-divisor.
\begin{definition}[$\sigma$-dimension]
Let $X$ be a projective variety.
Let $D=\sum a_i C_i$ be an $\mathbf{R}$-Cartier $\mathbf{R}$-divisor, where $a_i\in \mathbf{R}$ and $C_i$'s are integral Cartier divisor and let $H$ be any integral Cartier divisor. 
We then define
\begin{multline*}
\kappa_\sigma(D)\\
:=\max_{H \text{ integral Cartier}}\{\max\{l\in\mathbb{Z}|\,\limsup_{t\rightarrow \infty}\frac{h^{0}(X,\mathscr{O}_X(\sum\lfloor ta_i\rfloor C_i +H))}{t^l}>0\}\}.
\end{multline*}
\end{definition}
This is a measure of positivity of an $\mathbf{R}$-Cartier $\mathbf{R}$-divisor that lies on the boundary of the pseudoeffective cone.
However, this definition looks slightly different from the one that appeared in the literature (\cite{Nak04},\cite{Leh13} and \cite{Eck16}).
We shall prove in Proposition \ref{prop: numdim} that the definition is well-posed, i.e. independent of the decomposition $D=\sum a_i C_i$; is a numerical invariant and agrees with the usual definition when $X$ is smooth.
Nakayama's proof of the fact that $\sigma$-dimension is a numerical invariant relies on an Angehrn-Siu type argument, which requires smoothness of $X$.
On a singular projective variety $X$, it is possible to define the $\sigma$-dimension of a class $\eta\in\N^{1}(X)$ via the following way, due to Lehmann \cite[Chapter 6.1]{Leh13}.
Take a resolution of singularities of $X$, $\pi:\tilde{X}\rightarrow X$, define $\kappa_\sigma(\eta):=\kappa_\sigma(\pi^*\eta)$ and note that on smooth projective varieties the $\sigma$-dimension is a birational invariant \cite[Proposition V.2.7]{Nak04}.

The proof of the following lemma was suggested by the referee.
\begin{lemma}\label{lemma: global section}
Let $X$ be a projective variety.
Let $\mathscr{B}\subset \N^{1}(X)_{\mathbf{R}}$ be a bounded subset.
Let $H$ be an ample divisor on $X$,
Then for $m\gg 0$,
\[
H^{0}(X,\mathscr{O}_X(mH-C))\neq 0,
\]
for any integral Cartier divisor $C$ with $[C]\in \mathscr{B}$.
\end{lemma}
\begin{proof}
We prove by induction on $\dim X=n$.
This is true if $\dim X=1$ by the Riemann-Roch theorem.

Take a general hyperplane section in $H_1\in|m_1 H|$ for $m_1\gg 0$. 
It is irreducible and reduced.
Consider the short exact sequence
\[
0\rightarrow
\mathscr{O}_X(m_2 H -C)
\rightarrow
\mathscr{O}_X((m_1+m_2)H-C)
\rightarrow
\mathscr{O}_{H_1}((m_1+m_2)H-C)
\rightarrow 0.
\]
By induction,
$h^{0}(H_1,\mathscr{O}_{H_1}((m_1+m_2)H-C))\neq 0$ for $m_2\gg 0$ and for any integral Cartier divisor $C$ with $[C]\in \mathscr{B}$.
By Fujita vanishing theorem,
$h^{1}(X,\mathscr{O}_X(m_2 H -C))=0$ for $m_2\gg 0$ and for any integral Cartier divisor $C$ with $[C]\in \mathscr{B}$.
These imply that $h^{0}(X,\mathscr{O}_X((m_1+m_2)H-C))\neq 0$ for $m_2\gg 0$ and for any integral Cartier divisor $C$ with $[C]\in \mathscr{B}$.


\end{proof}

\begin{proposition}\label{prop: numdim}
Let $X$ be a projective variety and let $D$ be a pseudoeffective $\mathbf{R}$-Cartier $\mathbf{R}$-divisor on $X$. Then
\begin{enumerate}
    \item\label{item: numerical} The definition of $\kappa_\sigma(D)$ does not depend on the decomposition $D=\sum a_i C_i$. In fact, if $D\equiv D'$, then $\kappa_\sigma(D)=\kappa_\sigma(D')$.
    \item\label{item: numerical2}
    Assuming that $X$ is smooth,
    \begin{multline*}
    \kappa_{\sigma}(D)\\
=\max_{H \text{integral Cartier}}\{\max\{
l\in \mathbb{Z}|\,\limsup_{m\rightarrow\infty}\frac{h^{0}(X,\mathscr{O}_X(\lfloor mD\rfloor+H))}{m^{l}}>0
\}\}.
\end{multline*}
The right hand side of this equation is the usual definition of the $\kappa_\sigma(D)$ (\cite{Nak04},\cite{Leh13},\cite{Eck16}). 
Here we are rounding down $D$ as an $\mathbf{R}$-Weil divisor.
\end{enumerate}
\end{proposition}
\begin{proof}
For (\ref{item: numerical}),
suppose $D\equiv D'$, $D=\sum a_i C_i$ and $D'=\sum a'_i C'_i$.
By lemma \ref{lemma: global section}, there is an integral Cartier divisor $H'$ such that $\mathscr{O}_X(H'+C)$ is effective for any integral Cartier $C\equiv\sum r_i C_i + \sum r'_j C'_j$ where $r_i, r'_j\in [-2,2]$.
Given any integral Cartier divisor $H$, write $\sum \lfloor ma'_i\rfloor C'_i +H+H'$ as
\[
\sum \lfloor ma_i\rfloor C_i+H+(\sum \lfloor ma'_i\rfloor C'_i - mD') + (mD-\sum \lfloor ma_i\rfloor C_i)+(mD'-mD) +H'.
\]
This implies $h^{0}(X,\mathscr{O}_X(\sum \lfloor ma_i\rfloor C_i +H))\leq h^{0}(X,\mathscr{O}_X(\sum \lfloor ma'_i\rfloor C'_i +H+H'))$.
We can reverse the roles of $D$ and $D'$ and conclude (\ref{item: numerical}).

For (\ref{item: numerical2}),
$D$ is expressed uniquely as $\sum a_i \Gamma_i$, where $\Gamma_i$'s are prime divisors (which are Cartier by the smoothness assumption), $a_i\in \mathbf{R}$.
We have $\lfloor mD\rfloor=\sum \lfloor ma_i\rfloor \Gamma_i$, the equality then follows from (\ref{item: numerical}).

\end{proof}

Thanks to Proposition \ref{prop: numdim} (\ref{item: numerical}), we may refer to $\kappa_{\sigma}(\eta)$, where $\eta\in\N^{1}(X)_{\mathbf{R}}$, without ambiguity.

Here are some of the basic properties of $\kappa_\sigma(D)$.
The proofs of (\ref{it: pullback}) and (\ref{it: big}) are essentially the same as the ones given in \cite[Proposition V.2.7]{Nak04}.
\begin{proposition}[Basic properties\label{prop:basic}]
Let $X$ be a projective variety of dimension $n$ and let $\eta\in\N^{1}(X)_{\mathbf{R}}$.
\begin{enumerate}
    \item \label{it: pullback}
    If $f:X'\rightarrow X$ is a surjective morphism from a projective variety, then $\kappa_\sigma(\eta)=\kappa_\sigma(f^{*}(\eta))$.
    \item \label{it: range}
    $\kappa_\sigma(\eta)\leq n$.
    \item \label{it: range 2}
    $\kappa_\sigma(\eta)\geq 0$ if and only if $\eta$ is pseudoeffective.
    \item \label{it: big}
    $\kappa_\sigma(\eta)=n$ if and only if $\eta$ is big.
\end{enumerate}
\end{proposition}
\begin{proof}
Let $D=\sum a_i C_i$ be an $\mathbf{R}$-Cartier $\mathbf{R}$-divisor on $X$, such that the numerical class of $D$ is $\eta$.
For (\ref{it: pullback}), we let $H$ be an ample divisor on $X'$.
First, we claim that $f_{*}\mathscr{O}_{X'}(H)$ is torsion-free.
Since $f$ is surjective, the natural map
$\mathscr{O}_X\rightarrow f_{*}\mathscr{O}_{X'}$ is an injection.
Say a section $s\in f_{*}\mathscr{O}_{X'}(H)(U)$, $s\neq 0$, is torsion, i.e. there is $r\in\mathscr{O}_X(U)$, $r\neq 0$ such that $r\cdot s=0$.
Here $U$ is an open subset of $X$.
But $r$ and $s$ can be identified as nontrivial sections of $\mathscr{O}_{X'}(f^{-1}(U))$ and $\mathscr{O}_{X'}(H)(f^{-1}(U))$ respectively.
This contradicts the fact that $\mathscr{O}_{X'}(H)$ is invertible.
Next, as $f_{*}\mathscr{O}_{X'}(H)$ is torsion-free,
the canonical map $f_{*}\mathscr{O}_{X'}(H)\rightarrow (f_{*}\mathscr{O}_{X'}(H))^{\vee\vee}$ is injective.
There is some ample divisor $A$ on $X$ such that
we have the following surjection
$\oplus_{1}^k \mathscr{O}_X(-A)\twoheadrightarrow (f_{*}\mathscr{O}_{X'}(H))^{\vee}$.
Dualizing, this gives an injection $f_{*}\mathscr{O}_{X'}(H)\hookrightarrow
\oplus_{1}^k \mathscr{O}_X(A)$.
Hence,
$h^{0}(X',\mathscr{O}_{X'}(\sum\lfloor ta_i\rfloor f^{*}C_i +H))
\leq
k\cdot h^{0}(X,\mathscr{O}_{X}(\sum\lfloor ta_i\rfloor C_i +A))
$
and $\kappa_\sigma(\eta)\geq \kappa_\sigma(f^{*}(\eta))$.
The other direction is obvious.

For (\ref{it: range}), take a sufficiently ample divisor $H$ that computes the $\kappa_{\sigma}(D)$ and that $H-D$ is ample.
Then $h^{0}(X,\mathscr{O}_{X}(\sum\lfloor ta_i\rfloor C_i +H))\leq h^{0}(X,\mathscr{O}_{X}(\lfloor t +1 \rfloor H))$ for $t\gg 0$.
It follows that $\kappa_{\sigma}(D)\leq n$.

For (\ref{it: range 2}), 
if $\kappa_{\sigma}(D)\geq 0$,
then there is some divisor $H$ and a sequence $t_j\rightarrow \infty$ such that
$h^{0}(X,\mathscr{O}_X(\sum_i \lfloor t_j a_i\rfloor C_i+H))\neq 0$.
Write
\[
D=\frac{1}{t_j}(\sum_i \lfloor t_j a_i\rfloor C_i+H)+\frac{1}{t_j}(\sum_i (t_j a_i-\lfloor t_j a_i\rfloor) C_i-H).
\]
We observe that the first term on the right hand side is effective and the second term goes to $0$ as $t_j\rightarrow \infty$.
Thus, $D$ is pseudoeffective.

Now assume that $\kappa_{\sigma}(D)< 0$,
i.e.
for any divisor $H'$, $h^{0}(X,\mathscr{O}_{X}(\sum\lfloor t a_i\rfloor C_i +H'))= 0$ for all $t\gg 0$,
we would like to show that $D$ is not pseudoeffective.
By taking a sufficiently large multiple of an ample divisor,
we can find a Koszul-ample divisor $H$ such that $H+\sum e_i C_i$ is ample for any $e_i\in [-1,1]$.
By Lemma \ref{lemma:dualizing}, we can find an embedding of the dualizing sheaf of $X$,
$\omega_X\hookrightarrow \mathscr{O}_X(mH)$ for some large $m$.
By Serre duality,
$h^n(X,\mathscr{O}_X(-\sum\lfloor t a_i\rfloor C_i-H)\otimes \mathscr{O}_X((-n-1)H))
=h^0(X,\omega_X\otimes \mathscr{O}_X(\sum\lfloor t a_i\rfloor C_i +H)\otimes \mathscr{O}_X((n+1)H))
\leq
h^0(X,\mathscr{O}_X(mH)\otimes\mathscr{O}_X(\sum\lfloor t a_i\rfloor C_i +H)\otimes \mathscr{O}_X((n+1)H))$,
which is $0$ for $t\gg 0$ by assumption.

This shows that $-\sum\lfloor t a_i\rfloor C_i-H$ is $(n-1)$-ample for $t\gg 0$ by theorem \ref{thm:q-T-ample} 
and
implies that $\sum\lfloor t a_i\rfloor C_i+H= tD+(H-\sum (ta_i-\lfloor t a_i\rfloor) C_i)$ is not pseudoeffective by theorem \ref{theorem: n-1 ample}, hence $D$ is not pseudoeffective as well.

For (\ref{it: big}), if $D$ is big, it is clear that $\kappa_{\sigma}(D)=n$ by (\ref{it: range}).
Now assume that $\kappa_{\sigma}(D)=n$.
Let $H$ be a sufficiently ample divisor that computes $\kappa_{\sigma}(D)$.
We may find some $m$ such that $mH$ is very ample, and that $(m-1)H+\sum c_i C_i$ is ample for any $c_i\in[0,1]$.
By Bertini's theorem, we can find an irreducible, reduced and effective divisor $H_1$ that is rationally equivalent to $mH$.
Consider the following short exact sequence
\begin{multline*}
0\rightarrow
\mathscr{O}_{X}((\sum\lfloor t a_i\rfloor C_i) +(1-m)H)
\rightarrow
\mathscr{O}_{X}((\sum\lfloor t a_i\rfloor C_i) +H)\\
\rightarrow
\mathscr{O}_{X}((\sum\lfloor t a_i\rfloor C_i) +H)|_{H_1}
\rightarrow
0.
\end{multline*}
We may find a sequence of $t_j\rightarrow\infty$ and some $c>0$ such that
$h^{0}(X,\mathscr{O}_{X}((\sum\lfloor t_j a_i\rfloor C_i) +H))\geq ct_j^n$.
But by (\ref{it: range}), $h^{0}(H_1,\mathscr{O}_{X}((\sum\lfloor t_j a_i\rfloor C_i) +H)|_{H_1})< ct^n$ for $t\gg 0$.
These imply that $(\sum\lfloor t_j a_i\rfloor C_i) -(m-1)H$ is effective for $j\gg 0$.
Hence,
$t_jD= 
((\sum\lfloor t_j a_i\rfloor C_i) -(m-1)H)+((m-1)H+\sum (t_j a_i-\lfloor t_j a_i\rfloor) C_i)$ is big.
\end{proof}

\subsection{Ample and Locally ample subvarieties}
In this subsection, we shall first recall the definition of an ample subsubscheme, which was introduced by Ottem in \cite{Ott12}.
Then we introduce the notion of a locally ample subscheme, which generalizes the notion of a subvariety that is l.c.i. in the ambient variety with ample normal bundle.
\begin{definition}[Ample subscheme {\cite[Definition 3.1]{Ott12}}]
Let $X$ be a projective scheme of dimension $n$ and let $Y$ be a subscheme of $X$ of codimension $r$.
Let $E$ be the exceptional divisor of the blowup of $X$ along $Y$.
We say that $Y$ is an \textit{ample subscheme} of $X$ if $E$ is $(r-1)$-ample.
\end{definition}

This notion of ample subschemes indeed generalize the notion of an ample divisor naturally.
For example, if $Y$ is a smooth ample subvariety of a smooth projective variety,
then the Lefschetz hyperplane theorem with rational coefficient holds:
the natural maps
\[
H^{i}(X,\mathbb{Q})\rightarrow H^{i}(Y,\mathbb{Q})
\]
are isomorphisms for $i<n-r$ and is injective for $i=n-r$ \cite[Corollary 5.3]{Ott12}.

From the point of view of intersection theory,
we also know that if $Y$ is an l.c.i. ample subvariety of a projective variety $X$. Then for any subvariety $Z$ of $X$ of complementary dimension, $Y\cdot Z>0$ \cite{FL83}.

For more about ample subvarieties, c.f. \cite{Ott12}.

\begin{definition}[Locally ample subscheme]
Let $X$ be a projective scheme of dimension $n$ and let $Y$ be a subscheme of $X$ of codimension $r$.
Let $E$ be the exceptional divisor of the blowup of $X$ along $Y$.
We say that $Y$ is an \textit{locally ample subscheme} of $X$ if $\mathscr{O}_E(E)$ is $(r-1)$-ample.
\end{definition}
The following proposition shows that the concept of a locally ample subscheme generalizes the notion of an l.c.i. subvariety with ample normal bundle.
\begin{proposition}\cite[Corollary 4.3]{Ott12}
Let $X$ be a projective scheme of dimension $n$ and let $Y$ be a l.c.i. subscheme of $X$ of codimension $r$.
Then $Y$ has ample normal bundle if and only if $Y$ is locally ample in $X$.
\end{proposition}

\begin{proposition}[Pullback]\label{prop:pullback}
Let $X$ be a projective scheme and let $Y$ be a locally ample subscheme of $X$ of codimension $r$.
Let $Z$ be a closed subscheme of $X$.
Suppose $Y\cap Z$ has codimension $r$ in $Z$.
Then $Y\cap Z$ is locally ample in $Z$.
\end{proposition}
\begin{proof}
Indeed, by the universal property of blowup, we have the following commutative diagram
\[
\xymatrix{
\Bl_{Y\cap Z} Z \ar@{^{(}->}[r] \ar[d]_{\pi_{Z}}
& \Bl_{Y}X \ar[d]^{\pi_X}\\
Z \ar@{^{(}->}[r] & X.
}
\]

Note that the exceptional divisor of $\pi_Z$, $E_Z$, is the restriction of the exceptional divisor $E$ of $\pi_X$.
If $\mathscr{O}_E(E)$ is $(r-1)$-ample, so is $\mathscr{O}_{E_Z}(E)$. 
\end{proof}
We now show that the notion of locally ample subscheme satisfies the transitivity property.
The proof is a bit involved but is very similar to the proof of transitivity of ample subschemes \cite[Theorem 4.10]{Lau16},
it will be given in the appendix.
The following theorem on transitivity hints that the notion of locally ample subvarieties is a reasonable generalization of the notion of subvarieties with ample normal bundle.
However, we won't need it later.
\begin{theorem}[Transitivity of locally ample subschemes]\label{thm:transitivity}
Let $Y$ be a locally ample subscheme of $X$ of codimension $r_1$ and let $Z$ be a locally ample subscheme of $Y$ of codimension $r_2$.
Then $Z$ is a locally ample subscheme of $X$ of codimension $r_1 + r_2$.
\end{theorem}
\begin{corollary}[Intersection of locally ample subschemes]
Let $X$ be a projective scheme.
Let $Y$ and $Z$ be locally ample subschemes of $X$ of codimension $r$ and $s$ respectively and that $Y\cap Z$ is of codimension $r+s$ in $X$.
Then $Y\cap Z$ is locally ample in $X$.
\end{corollary}
\begin{proof}
By Proposition \ref{prop:pullback}, $Y\cap Z$ is locally ample in $Z$.
Hence, $Y\cap Z$ is locally ample in $X$ as well.
\end{proof}

\section{Numerical dominance}
In this section, we prove a basic fact on Nakayama's notion of numerical dominance,
which will streamline the argument in the proof of the main theorem.

Let us first start by stating the definition of numerical dominance.
\begin{definition}\cite[Definition 2.12]{Nak04}
Given two classes $\eta_1, \eta_2\in\N^{1}(X)_{\mathbf{R}}$.
We say that $\eta_1$ \textit{numerically dominates} $\eta_2$ if
for any ample divisor $A$ and for any $b\in \mathbf{R}$ there are $t_1, t_2>b$ such that $t_1 \eta_1 -t_2 \eta_2 +A$ is pseudoeffective.

We say that a class $\eta\in\N^{1}(X)_{\mathbf{R}}$ numerically dominates a closed subvariety $Y$ of $X$ if
on the blowup $\pi:\Bl_Y X \rightarrow X$, $\pi^*\eta$ numerically dominates the exceptional divisor $E$. 
\end{definition}

\begin{lemma}\label{lemma:numdom}
Let $X$ be a projective variety and
let $\eta_1,\eta_2\in\N^{1}(X)_{\mathbf{R}}$.
Then $\eta_1$ numerically dominates $\eta_2$ if and only if \underline{there exists} an ample divisor $A$ such that for any $b\in \mathbf{R}$ there are $t_1, t_2>b$ such that $t_1 \eta_1 -t_2 \eta_2 +A$ is pseudoeffective.
\end{lemma}
\begin{proof}
Suppose the hypothesis in the lemma holds.
Given an ample divisor $A'$, choose a large enough integer $a$ such that $aA'-A$ is pseudoeffective.
Given $b>0$, take $t_1, t_2 >ab$ such that $t_1 \eta_1 -t_2 \eta_2 +A$ is pseudoeffective.
Then $\frac{t_1}{a} \eta_1 -\frac{t_2}{a} \eta_2 +A'=\frac{1}{a}(t_1 \eta_1 -t_2 \eta_2 +A)+(A'-\frac{1}{a}A)$ is pseudoeffective.
\end{proof}

Let us relate the negation of numerical dominance and vanishing of the top cohomology group.
\begin{proposition}\label{prop:topcoh}
Let $X$ be a projective variety of dimension $n$ and let $Y$ be a subvariety of $X$.
Let $E$ be the exceptional divisor on $\tilde{X}:=\Bl_Y X$, the blowup of $X$ along $Y$.
Let $D$ be a pseudoeffective $\mathbf{R}$-Cartier $\mathbf{R}$-divisor on $X$,
written as $\sum a_i C_i$, 
where $a_i\in\mathbf{R}$ and $C_i$'s are integral Cartier divisors.
Fix a $2n$-Koszul-ample line bundle $\mathscr{O}(H)$ on $\tilde{X}$.

Suppose there is some $b\in\mathbf{R}$ such that
\[
h^{n}(\tilde{X},\mathscr{O}_{\tilde{X}}(kE-(\sum\lfloor ta_i\rfloor\pi^*C_i) -(m+n+2)H))=0
\]
for all $t\in(b,+\infty)$ and for all integer $k>b$, 
where $m\in \mathbb{N}$ such that $mH+ eE-\sum c_i\pi^*C_i$ is ample for any $e,c_i\in[0,1]$ on $\tilde{X}$,
then $D$ does not numerically dominate $Y$.

On the other hand, if $D$ does not numerically dominate $Y$, then
for any divisor $B$, there is $b\in \mathbf{R}$ such that 
\[
h^{n}(\tilde{X},\mathscr{O}_{\tilde{X}}(kE-(\sum\lfloor ta_i\rfloor\pi^*C_i) -B))=0
\]
for all $t\in(b,+\infty)$ and for all integer $k>b$.
\end{proposition}
\begin{proof}
For the first statement,
by the hypothesis,
\[
h^{n}(\tilde{X},\mathscr{O}_{\tilde{X}}(kE-(\sum\lfloor ta_i\rfloor\pi^*C_i) -(m+1)H)\otimes\mathscr{O}_{\tilde{X}}(-(n+1)H))=0
\]
for $k,t>b$, $k\in\mathbb{Z}$.
Thus, by Theorem \ref{thm:q-T-ample},
$kE-(\sum\lfloor ta_i\rfloor\pi^*C_i) -(m+1)H$ is $(n-1)$-ample for $k,t>b$, $k\in\mathbb{Z}$.
For $t_1, t_2>b$, we can write $t_2 E - t_1 \pi^*D- H
=(\lfloor t_2 \rfloor E- (\sum\lfloor t_1a_i\rfloor \pi^*C_i) -  (m+1)H)+ 
(mH+(t_2-\lfloor t_2 \rfloor )E-\sum(t_1 a_i-\lfloor t_1 a_i\rfloor)\pi^*C_i)$
and observe that the first term is $(n-1)$-ample and the second term is ample.
It follows that $t_2 E - t_1 \pi^*D- H$ is $(n-1)$-ample for all $t_1, t_2>b$.
Thus, $t_1 \pi^*D- t_2 E + H$ is not pseudoeffective for all $t_1, t_2>b$.
This proves the first assertion.

For the second statement,
for sufficiently large $l$, we can embed $\omega_{\tilde{X}}\hookrightarrow \mathscr{O}(lH)$ (Lemma \ref{lemma:dualizing}).
We may also assume that $B+lH$ is ample.
Take $m\in \mathbb{N}$ such that
$mH+\sum c_iC_i$ is ample for any $c_i\in [0,1]$.
By Lemma \ref{lemma:numdom}, there is a $b$ such that $t_1\pi^*D - t_2 E+B+(l+m)H$ is not pseudoeffective for $t_1, t_2>b$.
Thus,
for $k,t>b$ and $k\in\mathbb{Z}$,
\begin{align*}
&h^{n}(\tilde{X},\mathscr{O}_{\tilde{X}}(kE-(\sum \lfloor ta_i\rfloor\pi^*C_i) -B))&\\
&=h^{0}(\tilde{X},\omega_{\tilde{X}}\otimes\mathscr{O}_{\tilde{X}}((\sum \lfloor ta_i\rfloor\pi^*C_i) -kE +B)) &(\text{Duality})\\
&\leq
h^{0}(\tilde{X},\mathscr{O}_{\tilde{X}}((\sum \lfloor ta_i\rfloor\pi^*C_i) -kE +B+lH)) &(\omega_{\tilde{X}}\hookrightarrow \mathscr{O}(lH))\\
\end{align*}
Writing
\begin{multline*}
(\sum \lfloor ta_i\rfloor\pi^*C_i) -kE +B+lH\\
=[t\pi^*D -kE +B+(l+m)H]
-[mH+\sum (ta_i- \lfloor ta_i \rfloor) \pi^*C_i]
\end{multline*}
and observe that the first term on the right hand side is not pseudoeffective while the second term is ample,
we see that
$h^{0}(\tilde{X},\mathscr{O}_{\tilde{X}}((\sum \lfloor ta_i\rfloor\pi^*C_i) -kE +B+lH))=0$.

\end{proof}
\begin{remark}
In general, the divisor $\sum\lfloor ta_i\rfloor\pi^*C_i$ appearing in the statement of Proposition \ref{prop:topcoh} is different from the integral part $\lfloor t\pi^*D\rfloor$ of $t\pi^{*}D$, 
which is only a Weil divisor.
Note that the expression $\sum\lfloor ta_i\rfloor\pi^*C_i$ depends not only on $D$ and $t$ but also on the decomposition $D = \sum a_i C_i$ expressing $D$ as an $\mathbb{R}$-linear combination of integral Cartier divisors.
\end{remark}

\section{Proof of Theorem A}
We are now ready to demonstrate how the notion of numerical dominance comes into the picture.

\begin{proposition}\label{prop:num dom}
Let $X$ be a projective variety of dimension $n$,
let $Y$ be a locally ample subvariety of codimension $r$ of $X$ and
let $\eta\in\N^{1}(X)_{\mathbf{R}}$ be a pseudoeffective class
such that $\eta|_{Y}$ is not big.
Then $\eta$ does not numerically dominate $Y$.
\end{proposition}
\begin{proof}
Let $\tilde{X}$ be the blowup of $X$ along $Y$, with exceptional divisor $E$.
We fix a Koszul-ample line bundle $\mathscr{O}_{\tilde{X}}(H)$.
Take $D=\sum a_i C_i$ to be an $\mathbf{R}$-Cartier $\mathbf{R}$-divisor representing $\eta$.
Here $a_i\in \mathbf{R}$ and $C_i$'s are integral Cartier divisors.
We fix an integer $l>n+1$ such that $(l-(n+1))H+eE-\sum c_i C_i$ is ample for any $e,c_i\in [0,1]$.

We would like to prove that for any coherent sheaf $\mathscr{F}$ on $E$, there is $k_0$ such that
\begin{equation}\label{eq:vanishingofcoh}
h^{n-1}(E,\mathscr{F}\otimes\mathscr{O}_E(kE -\sum \lfloor ta_i\rfloor \pi^*C_i-lH))=0
\end{equation}
for $k\geq k_0$ and $t\geq 0$.
It is enough to prove that for the vanishing of cohomology groups on each of the irreducible components of $E$.
In other words, letting $E'$ be an irreducible component of $E$,
it suffices to prove that there is $k'_0$ such that $h^{n-1}(E',\mathscr{F}\otimes\mathscr{O}_{E'}(kE -\sum \lfloor ta_i\rfloor \pi^*C_i-lH))=0$
for $k\geq k'_0$ and $t\geq 0$.
As there is a surjection $\oplus\mathscr{O}(B)\twoheadrightarrow \mathscr{F}$, where $\mathscr{O}(B)$ is a line bundle, 
it suffices to prove the vanishing assuming $\mathscr{F}$ is a line bundle $\mathscr{O}(B)$.
By duality,
\begin{multline*}
h^{n-1}(E',\mathscr{O}_{E'}(kE -\sum \lfloor ta_i\rfloor \pi^*C_i+B-lH))\\
= h^{0}(E',\omega_{E'}\otimes\mathscr{O}_{E'}(-kE +\sum \lfloor ta_i\rfloor \pi^*C_i-B+lH)),
\end{multline*}
where $\omega_{E'}$ is the dualizing sheaf of $E'$.
We may embed $\omega_{E'}\hookrightarrow\mathscr{O}_{E'}(jH)$ for some $j$ by lemma \ref{lemma:dualizing}.
It suffices to prove that there is $k'_0$ such that
\begin{equation}\label{eq:E'}
h^{0}(E',\mathscr{O}_{E'}(-kE +\sum \lfloor ta_i\rfloor \pi^*C_i -B+(l+j)H))=0
\end{equation}
for $k\geq k'_0$ and $t\geq 0$.

As $D|_Y$ is not big, $-D|_Y$ is $(n-r-1)$-almost ample.
By \cite[Proposition 2.8]{Lau16}, $\pi^*(-D)|_E$ is also $(n-r-1)$-almost ample.
Since $\mathscr{O}_E(E)$ is $(r-1)$-ample, 
we may take $k'_0$ such that $(kE+\sum e_i \pi^*C_i+B -(l+j)H)|_{E'}$ is $(r-1)$-ample for $k\geq k'_0$ and $e_i \in [0,1]$, thanks to the openness of the $(r-1)$-ample cone (Theorem \ref{thm:open}).
Thus for $k\geq k'_0$ and $t\geq 0$,
\begin{multline*}
(kE -\sum \lfloor ta_i\rfloor \pi^*C_i+B-(l+j)H)|_{E'}\\
=((kE+\sum \{ta_i\}\pi^* C_i+B -(l+j)H)+\pi^*(-tD))|_{E'}
\end{multline*}
is $(r-1)+(n-r-1)=(n-2)$-ample, by Theorem \ref{thm:open}.
Now we have (\ref{eq:E'}) by \cite[Theorem 9.1]{Tot13}, hence also (\ref{eq:vanishingofcoh}).

If we fix $t$ and take $k$ large enough,
then
$h^{n}(\tilde{X},\mathscr{O}_{\tilde{X}}(kE -\sum \lfloor ta_i\rfloor \pi^*C_i-lH)))=0$,
since $E$ is $(n-1)$-ample.
We tensor the short exact sequence
\begin{equation}\label{eq:ses}
0\rightarrow
\mathscr{O}_{\tilde{X}}(kE)\rightarrow
\mathscr{O}_{\tilde{X}}((k+1)E)\rightarrow
\mathscr{O}_{E}((k+1)E)\rightarrow
0
\end{equation}
by $\mathscr{O}_{\tilde{X}}(-\sum \lfloor ta_i\rfloor \pi^*C_i -lH)$,
and consider its associated long exact sequence of cohomologies.
We apply (\ref{eq:vanishingofcoh}), letting $\mathscr{F}$ to be the structure sheaf $\mathscr{O}_E$, there is $k_0$ such that $h^{n-1}(E,\mathscr{O}_E(kE -\sum \lfloor ta_i\rfloor \pi^*C_i-lH))=0$
for $k\geq k_0$ and $t\geq 0$.
Therefore, 
\[
h^{n}(\tilde{X},\mathscr{O}_{\tilde{X}}(kE -\sum \lfloor ta_i\rfloor \pi^*C_i -lH))=0
\]
for $k\geq k_0$ and $t\geq 0$.
We may now conclude the proof by applying Proposition \ref{prop:topcoh}.
\end{proof}

\begin{proposition}\label{prop:num dim}
Let $X$ be a projective variety and let $Y$ be a subvariety of $X$.
Let $D$ be a pseudoeffective $\mathbf{R}$-Cartier $\mathbf{R}$-divisor such that $D$ does not numerically dominate $Y$.
Let $\pi:\tilde{X}\rightarrow X$ be the blowup of $X$ along $Y$, with exceptional divisor $E$.
Suppose $\pi|_E:E\rightarrow Y$ is an equidimensional morphism.
Then $\kappa_{\sigma}(D)\leq \kappa_{\sigma}(D|_Y)$.
\end{proposition}
\begin{proof}
We use the same notations as in the proof of the preceding proposition.
By Proposition \ref{prop:basic}, $\kappa_{\sigma}(D)=\kappa_{\sigma}(\pi^*D)$.
It is enough to look at the growth (in $t$) of $h^{0}(\tilde{X},\mathscr{O}_{\tilde{X}}(\sum \lfloor ta_i\rfloor \pi^*C_i+b_1H))$, for a large enough integer $b_1$.
Since $\omega_{\tilde{X}}$ is generically a line bundle, the natural map $\mathscr{O}_{\tilde{X}}\rightarrow \omega_{\tilde{X}}^{\vee}\otimes \omega_{\tilde{X}}$ is an injection.
We have the inequality
\begin{align*}
&h^{0}(\tilde{X},\mathscr{O}_{\tilde{X}}(\sum \lfloor ta_i\rfloor \pi^*C_i+b_1H))\\
&\leq
h^{0}(\tilde{X},\omega_{\tilde{X}}^{\vee}\otimes \omega_{\tilde{X}}\otimes\mathscr{O}_{\tilde{X}}(\sum \lfloor ta_i\rfloor \pi^*C_i+b_1H))\\
&=
h^{n}(\tilde{X},\omega_{\tilde{X}}\otimes\mathscr{O}_{\tilde{X}}(-\sum \lfloor ta_i\rfloor \pi^*C_i-b_1H)).
\end{align*}
There is some surjection $\oplus^{N}\mathscr{O}_{\tilde{X}}(-b_2H)\twoheadrightarrow\omega_{\tilde{X}}$.
Therefore,
\begin{multline*}
h^{n}(\tilde{X},\omega_{\tilde{X}}\otimes\mathscr{O}_{\tilde{X}}(-\sum \lfloor ta_i\rfloor \pi^*C_i-b_1H))\\
\leq
N\cdot h^{n}(\tilde{X},\mathscr{O}_{\tilde{X}}(-\sum \lfloor ta_i\rfloor \pi^*C_i-(b_1+b_2)H))
\end{multline*}
By Proposition \ref{prop:num dom} and Proposition \ref{prop:topcoh}, there is $k_0$ such that 
\[
h^{n}(\tilde{X},\mathscr{O}_{\tilde{X}}(kE-\sum \lfloor ta_i\rfloor \pi^*C_i-(b_1+b_2)H))=0
\]
for $k\geq k_0$ and $t\geq k_0$.
Tensoring the short exact sequence (\ref{eq:ses}) by $\mathscr{O}_{\tilde{X}}(-\sum \lfloor ta_i\rfloor \pi^*C_i-(b_1+b_2)H)$ and considering the associated long exact sequence of cohomologies,
we have
\begin{multline*}
h^{n}(\tilde{X},\mathscr{O}_{\tilde{X}}(-\sum \lfloor ta_i\rfloor \pi^*C_i-(b_1+b_2)H))\\
\leq 
\sum_{k=1}^{k_0} h^{n-1}(E,\mathscr{O}_{E}(kE-\sum \lfloor ta_i\rfloor \pi^*C_i-(b_1+b_2)H))
\end{multline*}
for $t\geq k_0$.

Note that the restriction of $\pi:\tilde{X}\rightarrow X$ to the exceptional divisor $\pi|_E:E\rightarrow Y$ is an equidimensional morphism, with fiber dimension equals to $r-1$.
Thus, $R^{d}(\pi|_{E})_{*}\mathscr{O}_E(kE-(b_1+b_2)H)=0$ for $d>r-1$.
Note also that $\dim Y=n-r$, which implies that $h^{d}(Y,\mathscr{F})=0$ for $d>n-r$ and for any coherent sheaf $\mathscr{F}$ on $Y$.
We now apply Leray spectral sequence and the above remarks to see that
for $1\leq k \leq k_0$,
\begin{align*}
& h^{n-1}(E,\mathscr{O}_{E}(kE-\sum \lfloor ta_i\rfloor \pi^*C_i-(b_1+b_2)H))\\
&=
h^{n-r}(Y,(R^{r-1}(\pi|_E)_{*}\mathscr{O}_{E}(kE-(b_1+b_2)H))\otimes \mathscr{O}_{Y}(-\lfloor ta_i\rfloor C_i))
\\
&=
h^{0}(Y,\omega_Y\otimes(R^{r-1}(\pi|_E)_{*}\mathscr{O}_{E}(kE-(b_1+b_2)H))^{\vee}\otimes \mathscr{O}_{Y}(\lfloor ta_i\rfloor C_i)),
\end{align*}
where the last equality holds by Serre duality.
Since $(R^{r-1}(\pi|_E)_{*}\mathscr{O}_{E}(kE-(b_1+b_2)H))^{\vee}$ is reflexive \cite[Corollary 1.2]{Har80}
and by lemma \ref{lemma:dualizing},
for sufficiently large $l$,
there is an embedding
$\omega_Y\otimes(R^{r-1}(\pi|_E)_{*}\mathscr{O}_{E}(kE-(b_1+b_2)H))^{\vee}\hookrightarrow\oplus^{N_k} \mathscr{O}_{Y}(lH)$ for $1\leq k\leq k_0$.
We can conclude that
$h^{0}(\tilde{X},\mathscr{O}_{\tilde{X}}(\sum \lfloor ta_i\rfloor \pi^*C_i+b_1H))\leq N\cdot (\sum_{k=1}^{k_0}N_k)\cdot h^{0}(Y, \mathscr{O}_{Y}(\lfloor ta_i\rfloor C_i+lH))$ for $t\gg 0$.
This proves the proposition.
\end{proof}

\begin{customthm}{A}[Numerical dimension via restriction]\label{thm:num dim}
With the same assumptions as in proposition \ref{prop:num dom}.
Then $\kappa_{\sigma}(\eta)\leq \kappa_{\sigma}(\eta|_Y)$.
\end{customthm}
\begin{proof}
Combine Proposition \ref{prop:num dom} and \ref{prop:num dim} and note that if $Y$ is locally ample, then $E\rightarrow Y$ is equidimensional \cite[Proposition 4.6]{Lau16}.
\end{proof}
\section{Applications of Theorem \ref{thm: 3}}
We give two applications of Theorem \ref{thm: 3}.
The first one is on positivity of cycle classes of locally ample and ample curves; the second one concerns the fact that locally ample subvarieties cannot be contracted.

\subsection{Cycle classes of locally ample/ample curves}
Peternell conjectured that if $Y$ is a smooth curve with ample normal bundle in a smooth projective variety $X$ and $\eta\in\N^{1}(X)$ is a pseudoeffective class with $\eta|_Y= 0$, then $\kappa_{\sigma}(\eta)=0$ \cite[Conjecture 4.12]{Pet12}.
Ottem later showed that the conjecture is indeed true \cite[Theorem 1]{Ott16}.
From there, Peternell observed that the cycle class of a smooth curve with ample normal bundle lies in the interior of the cone of curves (\cite[Conjecture 4.1]{Pet12},\cite[Theorem 2]{Ott16}).
Indeed, if $\eta\in\N^{1}(X)_{\mathbf{R}}$ is nef and $\eta|_Y=0$, the conjecture says $\kappa_{\sigma}(\eta)=0$.
But this forces $\eta=0$.
We are able to generalize this result by removing any restrictions on smoothness of $X$ and $Y$.
\begin{proposition}\label{prop:ott}\cite{Ott16}
Let $X$ be a projective variety.
Let $\eta\in\N^{1}(X)_{\mathbf{R}}$ be a pseudoeffective class.
If $\kappa_{\sigma}(\eta)=0$ and $\eta$ is nef, 
then $\eta= 0$. 
\end{proposition}
\begin{proof}
It follows from the argument on \cite[p.5]{Ott16}.
We include the proof here for the sake of completeness.

Let $H$ be an ample divisor of $X$.
Note that if we can prove that $\eta|_H=0$, it would imply $\eta=0$.
By induction on dimension of $X$,
it suffices to show that $\kappa_{\sigma}(\eta|_H)=0$.
Let $D=\sum a_i C_i$ be a pseudoeffective $\mathbf{R}$-Cartier $\mathbf{R}$-divisor such that the numerical class of $D$ is $\eta$.
Here $a_i\in \mathbf{R}$ and $C_i$'s are integral Cartier divisors.
By Fujita vanishing theorem, there is a $k_1$ such that for $k\geq k_0$,
\[
H^{1}(X,\mathscr{O}_X(kH+ N)=0,
\]
for any nef divisor $N$.
Take a sufficiently large $k_1$ such that $k_1 H-\sum e_i C_i$ is ample, for any $e_i\in [0,1]$.
For $t\geq 0$, $k_1 H+\sum\lfloor ta_i\rfloor C_i=tD + (k_1 H - \sum \{ta_i\}C_i)$ is nef.
Thus,
\[
H^{1}(X,\mathscr{O}_X(kH+ \sum \lfloor ta_i\rfloor D)=0
\]
for $k\geq k_0+k_1$.
Therefore, we have the surjection
\[
H^{0}(X,\mathscr{O}_X(\sum \rfloor ta_i\lfloor D +kH)
\twoheadrightarrow
H^{0}(H,\mathscr{O}_H(\sum \rfloor ta_i\lfloor D +kH)
\]
for $k\geq k_0+k_1$ and $t\geq 0$.
Hence $\kappa_{\sigma}(\eta|_H)=0$.
\end{proof}

The following theorem generalizes the first half of the main theorem in Ottem's paper \cite[Theorem 2]{Ott16}.
\begin{theorem}
Let $X$ be a projective variety.
Let $Y$ be a locally ample subvariety of dimension $1$ of $X$.
Then the cycle class of $Y$ in $\N_1(X)_{\mathbf{R}}$ is big, i.e. it lies in the interior of the cone of curves, $\overline{\NE}(X)$.
\end{theorem}
\begin{proof}
Suppose there is some nef class $\eta \in \N^{1}(X)_{\mathbf{R}}$
such that $\eta|_Y= 0$.
By theorem \ref{thm:num dim},
$\kappa_{\sigma}(\eta)=0$.
We then apply Proposition \ref{prop:ott} to conclude that $\eta=0$.
\end{proof}
We shall need the following proposition which shows that a pseudoeffective class $\eta\in\N^{1}(X)_{\mathbf{R}}$ on a smooth projective variety with $\kappa_{\sigma}(\eta)=0$ is in fact ``effective''.
\begin{proposition}\cite[Proposition V.2.7]{Nak04}\label{prop:Nak}
Let $X$ be a smooth projective variety.
Let $\eta\in \N^{1}(X)_{\mathbf{R}}$ be a pseudoeffective class.
If $\kappa_{\sigma}(\eta)=0$,
then there is an $\mathbf{R}$-Cartier $\mathbf{R}$-divisor $\sum a_i C_i$,
where $a_i\in \mathbf{R}_{>0}$ and $C_i$ are prime divisors,
such that its numerical class in $\N^{1}(X)_{\mathbf{R}}$ equals to $\eta$.
\end{proposition}

We are now ready to show that the cycle class of an ample curve lies in the interior of the movable cone of curves.
This strengthens the second half of \cite[Theorem 2]{Ott16}.
\begin{theorem}
Let $X$ be a projective variety and
let $Y$ be a locally ample curve in $X$.
Suppose $Y$ meets all prime divisors of $X$.
Then the cycle class $[Y]$ lies in the interior of the movable cone of curves.
In particular, the cycle class of an ample subvariety of dimension $1$ lies in the interior of the movable cone of curves.
\end{theorem}
\begin{proof}
Note that the second statement follows from the first.
Indeed, if $Y$ is an ample curve in $X$,
then $H^{n-1}(X\backslash Y,\mathscr{F})=0$ for any coherent sheaf $\mathscr{F}$ on $X\backslash Y$ \cite[Proposition 5.1]{Ott12}.
In particular, $X\backslash Y$ cannot contain any prime divisor.

Let $\pi:\tilde{X}\rightarrow X$ be the blowup of $X$ along $Y$,
let $X'\xrightarrow{f'}\tilde{X}=\Bl_Y X$ be a resolution of singularities on $\tilde{X}$
and let $f=\pi\circ f'$ be the composition.
The famous result in \cite{BDPP} says that the dual cone of the movable cone of curves is the pseudoeffective cone.
We can apply \cite[Theorem 6.1]{Lau16} to see that $[Y]$ lies in the movable cone of curves.
It suffices to show that for any pseudoeffective class $\eta\in N^{1}(X)_{\mathbf{R}}$ such that $\eta \cdot [Y]=0$, then $\eta =0$.

Theorem \ref{thm:num dim} says that 
$\kappa_\sigma(f^*\eta)=\kappa_\sigma(\eta)=0$.
As $f^*\eta$ is pseudoeffective, it is equal to the class of an effective $\mathbf{R}$-Cartier $\mathbf{R}$ divisor $\sum b_i B_i$ where $b_i>0$ and $B_i$'s are prime divisors by proposition \ref{prop:Nak}.

Suppose $\bigcup \Supp(B_i)\cap f^{-1}(Y)= \varnothing$.
By the projection formula, $[\eta]\equiv\sum b_i f_{*}[B_i]$ in $N_{n-1}(X)$. 
But $\bigcup\Supp f(B_i)\cap Y=\varnothing$
and the hypothesis imply all $B_i$'s are exceptional.
Thus $[\eta]= 0$ in $N_{n-1}(X)$ and $\eta= 0$ by \cite[Example 2.7]{FL14}.

We may assume $\bigcup \Supp(B_i)\cap f^{-1}(Y)\neq \varnothing$.
Applying the negativity lemma to $\sum b_i B_i$ (note that $-\sum b_i B_i$ is clearly $f$-nef), for any closed point $p\in f(\bigcup \Supp (B_i))$,
$f^{-1}(p)\subset \bigcup\Supp(B_i)$.
Take a curve $C'\subset f^{-1}(Y)$ such that $f(C')=Y$.
By the previous remark,
$C' \cap \bigcup\Supp(B_i)\neq \varnothing$.
On the other hand,
$\sum b_i B_i\cdot [C']=f^{*}\eta\cdot [C']=\deg(\kappa(C):\kappa(Y))\eta\cdot [Y]=0$.
Therefore, $C'\subset \bigcup\Supp(B_i)$ and $f^{-1}(Y)\subset\bigcup \Supp(B_i)$.
Thus, $f^{\prime *}(\pi^*\eta -\epsilon E)$ is pseudoeffective for some small $\epsilon >0$.
But Proposition \ref{prop:num dom} says that $\eta$ does not dominate $Y$ numerically.
This gives a contradiction.
\end{proof}
\subsection{Locally ample subvarieties cannot be contracted}
In this subsection, we show that, as a consequence of Theorem \ref{thm: 3}, a locally ample subvariety cannot be contracted.
\begin{theorem}\label{thm: C}
Let $X$ be a projective variety and let $Y$ be a locally ample subvariety of $X$.
Suppose $f:X\rightarrow Z$ is a surjective morphism from $X$ to a projective variety $Z$.
Then if $\dim f(Y)<\dim Y$, then $f|_Y: Y\rightarrow Z$ is surjective, i.e. $f(Y)=Z$.
\end{theorem}
\begin{proof}
Let $A$ be an ample divisor on $Z$.
Then
$\dim f(Y)=\kappa_{\sigma}(A|_{f(Y)})=\kappa_{\sigma}(f^{*}(A)|_Y)<\dim Y$.
Note that $f^{*}(A)|_Y$ is not big.
By Theorem \ref{thm: 3},
\[
\kappa_{\sigma}(f^{*}(A))\leq \kappa_{\sigma}(f^{*}A|_Y).
\]
But $\kappa_{\sigma}(f^{*}(A))=\dim Z$.
This forces the equality $\dim Z=\dim f(Y)$.
\end{proof}

\begin{remark}
The special case of Theorem \ref{thm: C}, where $Y$ is contracted to a point, is observed by Ottem by an elementary argument \cite[Proof of Lemma 12]{Ott16}.
\end{remark}

\appendix
\section{Proof of theorem \ref{thm:transitivity}}
First, note that we have the following commutative diagram
\[
\xymatrix{
\Bl_{\mathscr{I}_{Y}\cdot\mathscr{I}_{Z}}X
\ar[r]^{\pi_Y} \ar[d]_{\pi_Z} \ar[dr]
&
\Bl_{\mathscr{I}_{Z}}X
\ar[d]^{\pi'_Z}
\\
\Bl_{\mathscr{I}_Y}X
\ar[r]_{\pi'_Y}
&
X,
}
\]
where $\pi_Y$ and $\pi_Z$ are induced by 
blowing up the ideals $\mathscr{I}_{Y}\cdot \mathscr{O}_{\Bl_{\mathscr{I}_Z}}$ and
$\mathscr{I}_Z\cdot \mathscr{O}_{\Bl_{\mathscr{I}_Y}}$ respectively.
Let $E'_Y$ and $E'_Z$ be the exceptional divisors of $\pi'_Y$ and $\pi'_Z$.
We also let $E_Z$ be the exceptional divisor of $\pi_Z$ and let $E_Y$ be the divisor in $\Bl_{\mathscr{I}_Y\cdot\mathscr{I}_Z}X$ such that $E_Y + E_Z$ is the exceptional divisor of $\pi_Y$.
Note that $\pi_Z^*E'_Y= E_Y + E_Z$ and $\pi_Y^*E'_Z=E_Z$.
The proof of the above statements can be found in \cite[Lemma 4.11]{Lau16}.

To prove that $Z$ is locally ample in $X$,
it is the same as to show that $\mathscr{O}_{E'_Z}(E'_Z)$ is $(r_1 +r_2 -1)$-ample.
If we let $\tilde{Y}$ be the strict transform of $Y$ in $\Bl_{\mathscr{I}_Z}X$.
We know that $\mathscr{O}_{E'_Z\cap \tilde{Y}}(E'_Z)$ is $(r_2-1)$-ample.
By \cite[Proposition 4.6]{Lau16},
we know that $\pi'_Y$ has fiber dimension at most $r_1 -1$.
Therefore, $\pi_Y$ has fiber dimension at most $r_1 -1$ as well.
Let $H$ be an ample divisor on $\Bl_{\mathscr{I}_Z}X$.
By \cite[Lemma 4.9]{Lau16},
it suffices to show that 
for any $l\geq 0$,
\[
H^{i}(E_Z,\mathscr{O}_{E_Z}(mE_Z)\otimes \pi_Y^*\mathscr{O}_{\Bl_{\mathscr{I}_Z}X}(-lH))=0
\]
for $i>r_1 +r_2 -1$ and $m\gg 0$.
Fix $l\in\mathbb{Z}_{\geq 0}$.

\begin{claim}
$(E_Z-\delta E_Y)|_{E_Z \cap E_Y}$ is $(r_2-1)$-ample for $0<\delta \ll 1$.
\end{claim}
\begin{proof}[Proof of claim]
Since $-E_Y$ is $\pi_Y$-ample,
$(\pi_{Y}^*E'_Z-\delta E_Y)|_{E_Z \cap E_Y}
=(E_Z-\delta E_Y)|_{E_Z \cap E_Y}$ is $(r_2 -1)$-ample for $0<\delta \ll 1$, by \cite[Proposition 2.8]{Lau16}.
\end{proof}

\begin{claim}
$(E_Y+E_Z-\epsilon E_Z)|_{E_Z}$ is $(r_1 -1)$-ample for $0<\epsilon \ll 1$.
\end{claim}
\begin{proof}[Proof of claim]
Note that
$\pi_Z$ restricts to a morphism $E_Z\rightarrow E'_Y$,
$(\pi_Z^*E'_Y - \epsilon E_Z)|_{E_Z}
=(E_Y+E_Z-\epsilon E_Z)|_{E_Z}$
is $(r_1-1)$-ample for $0<\epsilon\ll 1$ since $-E_Z$ is $\pi_Z$-ample, by \cite[Proposition 2.8]{Lau16}.
\end{proof}

By the above claims, for sufficiently large integer $k$,
$\mathscr{O}_{E_Z\cap E_Y}(kE_Z - E_Y)$ is $(r_2-1)$-ample and $\mathscr{O}_{E_Z}((k+1)E_Y+kE_Z)$ is $(r_1-1)$-ample.
Fix such $k$.

Given $m_1,m_2\in \mathbb{Z}$, write 
\[
m_1 E_Y+m_2 E_Z =\lambda_1(kE_Z - E_Y) +\lambda_2(kE_Z +(k+1)E_Y) + j_1 E_Y +j_2 E_Z,
\]
where $\lambda_2=\lfloor\frac{m_1+\lfloor  \frac{m_2}{k}\rfloor}{k+2}\rfloor$;
$\lambda_1= \lfloor\frac{m_2}{k}\rfloor-\lambda_2$;
$j_1=((m_1+\lfloor\frac{m_2}{k}\rfloor)\bmod (k+2))$
and
$j_2= (m_2\bmod k)$.
Note that $0\leq j_1<k+2$ and $0\leq j_2<k$.
The precise formulae for $\lambda_1$ and $\lambda_2$ are not very important.
The plan is to choose a big $m_2$, then let $m_1$ increases. 
As $m_1$ grows, $\lambda_1$ decreases and $\lambda_2$ increases.
We then use the positivity of $(r_2-1)$-ampleness of $\mathscr{O}_{E_Z\cap E_Y}(kE_Z- E_Y)$ and $(r_1-1)$-ampleness of $\mathscr{O}_{E_Z}(kE_Z + (k+1)E_Y)$ to prove the required vanishing statement.

Since $\mathscr{O}_{E_Z}(kE_Z +(k+1)E_Y)$ is $(r_1-1)$-ample, we may find $\Lambda_2$ such that 
\begin{equation}\label{eq: a}
H^{i}(E_Z,\mathscr{O}_{E_Z}(\lambda_2(kE_2 +(k+1)E_Y) + j_1 E_Y +j_2 E_Z)\otimes\pi_Y^*(\mathscr{O}_{\Bl_{\mathscr{I}_{Z}}X}(-lH)))=0
\end{equation}
for $i>r_1 -1$, $\lambda_2\geq \Lambda_2$, $0\leq j_1<k+2$ and $0\leq j_2 <k$.

Applying theorem \cite[Theorem 3.9]{Lau16} to the scheme $E_Z\cap E_Y$,
there is an $\Lambda'_2$ such that 
\begin{align*}
&H^{i}(E_Z\cap E_Y, \mathscr{O}_{E_Z\cap E_Y}(\lambda_1(kE_Z - E_Y)\\
& +\lambda_2(kE_2 +(k+1)E_Y) + j_1 E_Y +j_2 E_Z)\otimes \pi^{*}_Y\mathscr{O}_{\Bl_{\mathscr{I}_Z}}(-lH))\\
&=0
\end{align*}
for $i>(r_2 -1) + (r_1 -1)$, $\lambda_1\geq 0$, $\lambda_2 \geq \Lambda'_2$, $0\leq j_1<k+2$ and $0\leq j_2<k$.
This implies
\begin{multline}\label{eq: b}
H^{i}(E_Z,\mathscr{O}_{E_Z}(m_2 E_Z + m_1 E_Y)\otimes\pi_Y^*(\mathscr{O}_{\Bl_{\mathscr{I}_{Z}}X}(-lH)))\\
\cong
H^{i}(E_Z,\mathscr{O}_{E_Z}(m_2 E_Z + (m_1 +1)E_Y)\otimes\pi_Y^*(\mathscr{O}_{\Bl_{\mathscr{I}_{Z}}X}(-lH)))
\end{multline}
for $i>r_1 + r_2 -1$, $0<m_1+1< (k+1)\lfloor\frac{m_2}{k}\rfloor + k+2$ and $\lfloor\frac{m_1 +1+\lfloor  \frac{m_2}{k}\rfloor}{k+2}\rfloor\geq \Lambda'_2$.

Choose some big $M_2$ such that $\lfloor\frac{\lfloor  \frac{M_2}{k}\rfloor}{k+2}\rfloor\geq \max\{\Lambda_2,\Lambda'_2\}$.
Applying (\ref{eq: b}) repeatedly, we have for $m_2>M_2$, 
\begin{multline}
H^{i}(E_Z,\mathscr{O}_{E_Z}(m_2 E_Z)\otimes\pi_Y^*(\mathscr{O}_{\Bl_{\mathscr{I}_{Z}}X}(-lH)))\\
\cong
H^{i}(E_Z,\mathscr{O}_{E_Z}(m_2 E_Z + (k+1)\lfloor \frac{m_2}{k}\rfloor E_Y)\otimes\pi_Y^*(\mathscr{O}_{\Bl_{\mathscr{I}_{Z}}X}(-lH)))
\end{multline}
for $i>r_1 + r_2 -1$.
The above cohomology group can be rewritten as
\[
H^{i}(E_Z,\mathscr{O}_{E_Z}(\lfloor\frac{m_2}{k}\rfloor(kE_Z +(k+1)E_Y) +(m_2 -k\lfloor\frac{m_2}{k}\rfloor) E_Z)\otimes\pi_Y^*(\mathscr{O}_{\Bl_{\mathscr{I}_{Z}}X}(-lH))),
\]
which is $0$ by (\ref{eq: a}).
This completes the proof.


\begin{bibdiv}
\begin{biblist}
\bib{Bac86}{article}{
   author={Backelin, J{\"o}rgen},
   title={On the rates of growth of the homologies of Veronese subrings},
   conference={
      title={Algebra, algebraic topology and their interactions (Stockholm,
      1983)},
   },
   book={
      series={Lecture Notes in Math.},
      volume={1183},
      publisher={Springer, Berlin},
   },
   date={1986},
   pages={79--100},
}
\bib{BDPP}{article}{
   author={Boucksom, S{\'e}bastien},
   author={Demailly, Jean-Pierre},
   author={P{\u{a}}un, Mihai},
   author={Peternell, Thomas},
   title={The pseudo-effective cone of a compact K\"ahler manifold and
   varieties of negative Kodaira dimension},
   journal={J. Algebraic Geom.},
   volume={22},
   date={2013},
   number={2},
   pages={201--248},
}
\bib{DPS}{article}{
   author={Demailly, Jean-Pierre},
   author={Peternell, Thomas},
   author={Schneider, Michael},
   title={Holomorphic line bundles with partially vanishing cohomology},
   conference={
      title={Proceedings of the Hirzebruch 65 Conference on Algebraic
      Geometry },
      address={Ramat Gan},
      date={1993},
   },
   book={
      series={Israel Math. Conf. Proc.},
      volume={9},
      publisher={Bar-Ilan Univ., Ramat Gan},
   },
   date={1996},
   pages={165--198},
}
	
\bib{Eck16}{article}{
    author={Eckl, Thomas},
    title={Numerical analogues of the Kodaira dimension and the Abundance Conjecture},
    journal={Manuscripta Mathematica},
    volume={150},
    number={3},
    pages={337--356},
    year={2016}
}
\bib{FL83}{article}{
   author={Fulton, William},
   author={Lazarsfeld, Robert},
   title={Positive polynomials for ample vector bundles},
   journal={Ann. of Math. (2)},
   volume={118},
   date={1983},
   number={1},
   pages={35--60},
}
\bib{FL14}{article}{
    author={Fulger, Mihai},
    author={Lehmann, Brian},
    title={Positive cones of dual cycle classes},
   journal={Algebraic Geom.},
   volume={4},
   number={1},
   pages={1--28},
   date={2017}
}
\bib{GK15}{article}{
	author={Greb, Daniel},
	author={K\"uronya, Alex},
	title={Partial positivity: geometry and cohomology of $q$-ample line bundles},
	journal={London Math. Soc. Lecture Note Series},
	volume={417},
	pages={207-239},
	year={2015}
}
\bib{Har70}{book}{
   author={Hartshorne, Robin},
   title={Ample subvarieties of algebraic varieties},
   series={Lecture Notes in Mathematics, Vol. 156},
   note={Notes written in collaboration with C. Musili},
   publisher={Springer-Verlag, Berlin-New York},
   date={1970},
   pages={xiv+256},
}
\bib{Har77}{book}{
    author={Hartshorne, Robin},
    title={Algebraic Geometry},
    note={Graduate Texts in Mathematics, No. 52},
   publisher={Springer-Verlag, New York-Heidelberg},
   date={1977},
   pages={xvi+496},
}
\bib{Har80}{article}{
    author={Hartshorne, Robin},
    title={Stable reflexive sheaves},
    Journal={Math. Annalen},
    Volume = {254},
    Pages = {121--176},
    Year = {1980}
}
\bib{Lau16}{article}{
    author={Lau, Chung-Ching},
    title={Fujita vanishing theorems for q-ample divisors and applications on subvarieties with nef normal bundle},
    date={2016}
}
\bib{Leh13}{article}{
    author={Lehmann, Brian},
    title={Comparing numerical dimensions},
    journal={Algebra Number Theory},
    volume={7},
    number={5},
    pages={1065--1100},
    date={2013}
}
\bib{Nak04}{book}{
    author={Nakayama, Noboru},
    Title = {Zariski-decomposition and abundance},
    Pages = {xiii + 277},
    Year = {2004},
    Publisher = {Tokyo: Mathematical Society of Japan}
}
\bib{Ott12}{article}{
   author={Ottem, John Christian},
   title={Ample subvarieties and $q$-ample divisors},
   journal={Adv. Math.},
   volume={229},
   date={2012},
   number={5},
   pages={2868--2887},
}
\bib{Ott16}{article}{
    author={Ottem, John Christian},
   title={On subvarieties with ample normal bundle},
   journal={J. Eur. Math. Soc. (JEMS)},
   volume={18},
   date={2016},
   number={11},
   pages={2459-2468},
}
\bib{Pet12}{article}{
    author={Peternell, Thomas},
    title={Compact subvarieties with ample normal bundles, algebraicity, and cones of cycles},
    journal={Michigan Math. J.},
    volume={61},
    number={4},
    year={2012},
    pages={875--889}

}
\bib{Som78}{article}{
	author={Sommese, Andrew J.},
	title={Submanifolds of Abelian varieties},
	journal={Math. Annalen},
	volume={233},
	pages={229-256},
	year={1978}

}
\bib{Tot13}{article}{
   author={Totaro, Burt},
   title={Line bundles with partially vanishing cohomology},
   journal={J. Eur. Math. Soc. (JEMS)},
   volume={15},
   date={2013},
   number={3},
   pages={731--754},
}
\end{biblist}
\end{bibdiv}
\end{document}